\documentclass[11pt, oneside]{amsart}
\usepackage[a4paper, total={5.1in, 8.5in}]{geometry}
\usepackage{sidecap}
\usepackage{float}

\usepackage{color}

\usepackage{enumitem}
\newlist{myQuoteEnumerate}{enumerate}{2}
\setlist[myQuoteEnumerate,1]{label=(\alph*)}
\setlist[myQuoteEnumerate,2]{label=(\alph*)}

\usepackage[usenames,dvipsnames,svgnames,table]{xcolor}
\definecolor{awesome}{rgb}{1.0, 0.13, 0.32}

\definecolor{airforceblue}{rgb}{0.36, 0.54, 0.66}

\usepackage[dvipsnames]{xcolor}
\usepackage{amsmath,amssymb,latexsym}
\usepackage[utf8]{inputenc}
\usepackage[english]{babel}
\usepackage{booktabs}
\newtheorem{thm}{Theorem}[section]
\newtheorem{thmm}{Theorem}[section]
\newtheorem{mthmm}{Main Theorem}[section]

\newtheorem{lem}{Lemma}[section]

\newtheorem{defi}{Definition}[section]

\newcommand{\storus}{\mathbb{T}}

\newcommand{\inv}{\operatorname{inv}}

\usepackage{siunitx}
\usepackage{xcolor}
\usepackage{booktabs,colortbl, array}
\usepackage{pgfplotstable}

\definecolor{rulecolor}{RGB}{0,71,171}
\definecolor{tableheadcolor}{gray}{0.92}
\newcommand{\cdisk}{\overline{\mathbb{D}}}

\newcommand{\embd}{\operatorname{emb}}

\numberwithin{equation}{section}
\newcommand{\nid}{\noindent}

\newcommand{\disth}{\operatorname{dist}_{\textsc{h}}}
\newcommand{\hyp}{\operatorname{\textrm{hyp}}(A_r, c)}

\newcommand{\sh}{\operatorname{sh}}
\newcommand\quotient[2]{
        \mathchoice
            {
                \text{\raise1ex\hbox{$#1$}\Big/\lower1ex\hbox{$#2$}}%
            }
            {
                #1\,/\,#2
            }
            {
                #1\,/\,#2
            }
            {
                #1\,/\,#2
            }
    }

\definecolor{aurometalsaurus}{rgb}{0.43, 0.5, 0.5}
\definecolor{darkjunglegreen}{rgb}{0.1, 0.14, 0.13}
\definecolor{coolblack}{rgb}{0.0, 0.18, 0.39}
\definecolor{cobalt}{rgb}{0.0, 0.28, 0.67}

\newcommand{\ep}{\epsilon}
\usepackage{hyperref}
\hypersetup{
     colorlinks   = true,
     citecolor    = blue
}

\title{Holomorphic motions for unicritical correspondences}
\author{Carlos Siqueira}
\address{Departamento de Matem\'atica, ICMC-USP.  S\~ao Carlos -- SP, Brazil. Caixa Postal 668. CEP 13560-970.}
\curraddr[C.~Siqueira]{{\sc Imperial College London. Huxley Building, South Kensington Campus.
London SW7 2AZ, UK.}}
\email[C.~Siqueira]{carlos@icmc.usp.br, c.siqueira-lima15@imperial.ac.uk.}
\author{Daniel Smania}

\email[D.~Smania]{smania@icmc.usp.br.}

\date{} 

\begin{document}

\hypersetup{linkcolor=cobalt}

\begin{abstract} We study quasiconformal deformations and mixing properties of hyperbolic sets in the family of holomorphic correspondences $z^{r} +c,$ where $r >1$ is rational. Julia sets in this family are projections of Julia sets of holomorphic maps on  $\mathbb{C}^2,$ which are skew-products when $r$ is integer, and solenoids when $r$ is non-integer and $c$ is close to zero. Every hyperbolic Julia set in $\mathbb{C}^2$ moves holomorphically. The projection determines a branched holomorphic motion  with  local (and sometimes global) parameterisations of the plane Julia set by quasiconformal curves.  \end{abstract}

\maketitle

\keywords{MSC-class 2010:  37F15 (Primary) 32H50,  37F05 (Secondary).}


 \section{Introduction: definitions and main results} 
 
 The  quadratic family $z^2+c$ is one of the simplest examples of nonlinear dynamical systems which  exhibits a full spectrum of dynamical behaviour.     Other families $z^d +c$ have received more attention over the last years, with the remarkable discovery made by  \cite{LAS}  that  for Lebesgue almost every $c\in \mathbb{C},$ the map $z^d +c$ is either hyperbolic or infinitely renormalizable.
This paper is concerned with the study of  the family \begin{equation}\label{X}\textbf{f}_c(z)=z^{\beta} +c,\end{equation} where $c\in \mathbb{C}$, $\beta > 1$ is rational and $z^{\beta}=\exp \beta\, \textbf{log}(z).$

If $\beta=p/q$ in $(\ref{X})$ is presented in its reduced form, then every $z\in \mathbb{C}^*$ has exactly $q$ images under $\mathbf{f}_c.$ Hence $\mathbf{f}_c$ is a \emph{holomorphic correspondence,} i.e., a relation $z\mapsto w$ on the Riemann sphere determined implicitly by a complex polynomial equation $p(z,w)=0,$ which in this case is $(w-c)^q=z^p.$ We are going to consider the more general case where $q$ and $p$ are not necessarily relatively prime, with $p>q\geq 2.$ This is necessary, not only for the sake of generality, but also because $(w-c)^q=z^{p}$ is a conformal extension of the unimodal map \begin{equation}u_c(x)=|x|^{\beta} +c, \end{equation} when $\beta=p/q$ and both $p$ and $q$ are even (in the sense that the graph of $u_c$ is contained in the Riemann surface determined by the correspondence). Several problems in real dynamics have been successfully solved with the help of complex extensions. Indeed, it was only recently that Rempe and van Strien proved that the topological entropy of $f_a(x)=a\sin (\pi x)$ is a monotone function of $a$ using the the fact that $f_a$ extends to an entire transcendental map on the complex plane  \cite{RempeStrien}.  For  non-integer  exponents  like $\beta,$ it has been
 conjectured (for at least 30 years) that the topological entropy defines a monotone function $t\mapsto h_{top}(u_t).$

\subsection{Quasiconformal deformations}  Every finitely generated Kleinian group or rational function is an example of holomorphic correspondence.  Some of  the most interesting examples appear when a single polynomial equation $p(z,w)=0$ describes the \emph{mating} of Kleinian group with a polynomial map.  In  \cite{Bullett1994} we have  a family of matings parameterized on a set which (conjecturally) is homeomorphic to the Mandelbrot set of the quadratic family $z^2+c.$  In this direction, very recent results reveal a hybrid property  which is also present in $z^{\beta}+c.$  However, the main concern of this first paper on the dynamics of $z^{\beta}+c$ is to develop analytical tools to study quasiconformal conjugacy classes.

 Kleinian groups which are rigid as groups may be deformable when considered as holomorphic correspondences, as described by \cite{BP99}, where certain circle-packing Kleinian groups are proved to have families of deformations as quasicircle-packing correspondences. A similar phenomenon holds for quasifuchsian groups: {\it if $\Gamma$ is a finitely generated Kleinian group whose limit set is a Jordan curve, then $\Gamma$ is  quasiconformally conjugate to a Fuchsian group. Therefore, the limit set of $\Gamma$ is {quasicircle}}. For polynomial maps, the same principle applies: the Julia set of $z^q+c$ is a quasicircle for every $c$ sufficiently close to the origin. The Julia set $J(z^{\beta}+c)$ of the correspondence is the closure of repelling periodic orbits (\S \ref{obsec}).  In this paper we show: 
  
  \begin{thm}[Theorem \ref{kks}]\label{asdf} If $\beta \in \mathbb{Q}-\mathbb{Z}$ and $\beta > 1,$ then for every $c$ in a neighbourhood $U\subset \mathbb{C}$ of zero, the Julia set $J(z^{\beta}+c)$ can be  described as a  union of curves \begin{equation}\gamma: \mathbb{R} \to J(z^{\beta}+c)\end{equation} which are quasiconformal deformations of the universal cover of $\mathbb{S}^1.$

  \end{thm}

   \begin{figure}[H]\label{jdf}
\centering
\includegraphics[scale=0.2]{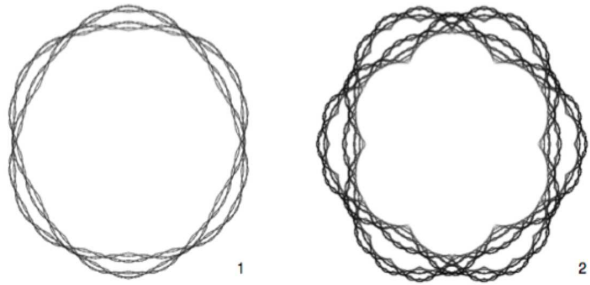}
{\caption{\it The Julia sets of $(w-c)^2=z^6$  for $(1)$ $c=0.2i$  and $(2)$  $c=0.35i.$ }}
\end{figure}

\subsection{Hyperbolic sets} By a hyperbolic set of $(w-c)^q=z^p$ we mean any compact subset $\Lambda$ of the plane such that $(i)$ every point of $\Lambda$ has at least one image in $\Lambda$  and at least one pre-image in $\Lambda$ under the correspondence;  and $(ii)$ the branches of the correspondence either expand or contract a conformal metric on $\Lambda$ (\S \ref{sect25}). Mixing properties for hyperbolic sets are discussed in \S \ref{vcb}. By Theorem \ref{jsd}, topologically  mixing hyperbolic sets often occur in  two special types: Julia sets $J_c$, defined as the closure of repelling cycles of $(w-c)^q=z^p;$  and dual Julia sets $J_c^*$, defined as the closure of attracting cycles.   In Figure \ref{jdf}, for example, the corresponding dual Julia sets  are Cantor sets $\{0, \ldots, p-1\}^{\mathbb{N}_0}$ clustered around the origin.  They are obtained from Conformal Iterated Function Systems induced by the restriction $\mathbf{f}_c: U \to U$  to a forward invariant neighbourhood $U$ of zero such that $\mathbf{f}_c^2(U)$ consists of $p$ disjoint disks avoiding the critical point. (This is an instance of a more general property concerning rigidity and wandering domains, currently under preparation).

 \subsection{Cantor bundles and holomorphic motions}  The standard method of obtaining holomorphic motions of hyperbolic Julia sets in families of rational maps uses the $\lambda$-Lemma of   \cite{Mane1983} after a preliminary motion of  repelling cycles. This method does not work for the family $(w-c)^q=z^p,$ but for a special type of hyperbolic sets called hyperbolic repellers
-- expanding hyperbolic sets   which are fully invariant under backward iterates -- \cite{Ochs} proved the structural  stability on the orbit space under the action of the (left) shift. However, we have hyperbolic sets which are not structurally stable in the plane.  Since ${J}_0= \mathbb{S}^1,$ Figure \ref{jdf}  gives a  preliminary  support to this fact, which becomes even more evident when we prove that $J_c$ is the projection of a solenoid.  

The proper ambient space to discuss structural stability and holomorphic motions is $\mathbb{C}^2.$ We may embed the space of forward orbits of $\mathbf{f}_0: \mathbb{S}^1 \to \mathbb{S}^1$   in $\mathbb{C}^2,$ obtaining a solenoid which is the Julia set of a holomorphic map $f_0: U\to \mathbb{C}^2$  defined on an open set $U \subset \mathbb{C}^2.$ This solenoid moves holomorphically in $\mathbb{C}^2$, and the projection of this holomorphic motion  determines a branched holomorphic motion of $J_c$ in the plane.  The  geometry  of $J_c$ will ultimately  depend on how the solenoid in $\mathbb{C}^2$ projects into $\mathbb{C}.$ For $c=0$ this projection is exceptionally simple.

Since  plane hyperbolic sets as defined in this paper are very general objects, it is surprising that under very weak assumptions such objects do in fact move holomorphically. For examples of attracting hyperbolic sets, dual Julia sets and hyperbolic Julia sets, see \S \ref{xvc}.  

 The motion in $\mathbb{C}^2$ associated to a hyperbolic set $\Lambda\subset \mathbb{C}$ is carried out with the introduction of a field of Cantor sets on $\Lambda.$ In other words, we assign a set $\mathcal{C}_z(\Lambda)$ -- which is very often a Cantor set -- at every point  $z\in \Lambda$ in a way that $\mathcal{C}_z(\Lambda)$ varies continuously on evenly covered subsets $K\subset \Lambda$ (see Theorem \ref{nossa}). The corresponding \emph{Cantor bundle} is \begin{equation}\mathcal{B}(\Lambda)= \bigcup_{z\in \Lambda} \{z\} \times \mathcal{C}_z(\Lambda). \end{equation}

    \begin{thm}[Short version of Theorem \ref{giocco}]\label{B}  Let $\Lambda_{c_0}$ be a hyperbolic set of the correspondence $(w-c_0)^{q}=z^p.$ There is a holomorphic motion based at $c_0,$ \begin{equation}h: U\times \mathcal{B}(\Lambda_{c_0}) \to \mathbb{C}^2 \end{equation} such that $\Lambda_c= \pi h_c(\mathcal{B}(\Lambda))$ is a hyperbolic set of $(w-c)^{q}=z^{p},$ for every $c\in U.$
  
  \end{thm}

\subsection{Julia sets in $\mathbb{C}^2$}Theorem \ref{B} has many consequences. For example, it is possible to show that there is a holomorphic map $f_c: U \to  \mathbb{C}^2$ defined on a neighbourhood of $\mathcal{B}(\Lambda_{c_0})$  such that  the Julia set  $J(f_c)$ of $f_c$ (Lemma \ref{semiconjugacy}) is invariant under $f_c.$ It can be shown that $\pi:\mathbb{C}^{2} \to \mathbb{C}$ establishes a semi-conjugacy from $f_c$ to $\mathbf{f}_c$ and $\pi(J(f_c)) = J(\mathbf{f}_c).$
When the Julia set is hyperbolic, Theorem  \ref{B} especializes to \begin{thm} \label{Cc}If the Julia set $J(\mathbf{f}_c)$  is hyperbolic  then   $J(f_{c_0})$ is contained in  $\mathcal{B}(\Lambda),$ \begin{equation}\pi J(f_c)=J_c(\mathbf{f}_c),\end{equation} and  the conjugacies  $h_c: J(f_{c_0}) \to J(f_c),  \ c\in U,$  define a holomorphic motion. \end{thm}

Theorem \ref{B} has an interesting meaning for parameters close to zero. It is possible to prove that $J_0 = \mathbb{S}^1$ is hyperbolic and $J(f_{0})$ is a solenoid.   Since $J(f_c)$ is the image of a holomorphic motion of $J(f_0),$ each $J(f_c)$ may be considered  a \emph{quasi-solenoid,} so that 
\begin{thm}[\bf Solenoids] \label{Dd}If $\beta >1$ is in $\mathbb{Q}-\mathbb{Z},$  then  for every $c$ in a neighbourhood of zero, the Julia set of $z^{\beta} +c$ is the projection of a quasi-solenoid in $\mathbb{C}^2.$ 

\end{thm} 

  The phenomenon described in Theorems \ref{Dd} and \ref{Cc} for the bundle maps $f_c$ bears some similarity with the result established by Hubbard and Oberste-Vorth  for H\'enon maps $g: \mathbb{C}^2 \to \mathbb{C}^2 $ given by \begin{equation}g_{c,b}(z,w)= (z^2 +c - bw, z). \end{equation}
      Indeed, if $q=1$ and $p=2,$ in which case $\mathbf{f}_c(z)=z^2+c,$ then the bundle map $f_c$ admits an extension $f_c: \mathbb{C}^2 \to \mathbb{C}^2$ given by the skew product \begin{equation}f_c(z,w) = (z^2+c,\,  \delta^{-1}w - rz^2 +rc ). \end{equation} It was proved by \cite{Hubbard95} that if  $p_c(z)=z^2+c$ is expanding on its Julia set and $|b|$ is sufficiently small,  then $g=g_{c,b}$ is hyperbolic on its Julia set $J(g)$ and $g:J(g) \to J(g)$ is topologically conjugate to the inverse limit of $p_c$ on its Julia set.  Recently  \cite{Smillie2010} proved that if $c=0$ and $|b| < (\sqrt{2} -1)/2,$ then $g:J(g) \to J(g)$ is topologically conjugate to a solenoid.

\subsection{Branched holomorphic motions} The projection of holomorphic motions given by Theorem \ref{B} determines branched holomorphic motions on the plane (\S \ref{pld}). This concept was  introduced by  \cite{Lyubich2015} for automorphisms of $\mathbb{C}^2.$  The quasiconformal properties of branched holomorphic motions obtained in this way are described in Theorem \ref{ofd}.

 The ultimate  consequence  of Theorem \ref{B}  should be the   application  of the ideas of   \cite{Ruelle} to prove the real analyticity of  the Hausdorff dimension of $J(f_c)\subset \mathbb{C}^2$ as a function of $c,$ which we believe is true, at least for $c$ close to zero.  
 
 {\it The main ideas of this paper may be easily generalized to other types of correspondences. }

   \subsection{Notation and terminology} \label{ghv} 
  Any connected open set is a \emph{region}.
 $$\mathbb{C}^* =\mathbb{C}-\{0\}, \  \hat{\mathbb{C}} =\mathbb{C}\cup \{\infty\}, \  \mathbb{N}_0=\{0,1, \cdots\}. $$ 
   $$ \mathbb{S}^1=\{z\in \mathbb{C}: |z|=1\},  \ \mathbb{T}=\mathbb{S}^{1}\times \overline{\mathbb{D}}.$$
  $\operatorname{card} (A)$ indicates the cardinality of a set $A.$ 
Every holomorphic correspondence is a multifunction. By a \emph{multifunction}  $\mathbf{f}: X\to Y$ we mean any relation $\mathbf{f} \subset X\times Y$ for which $$\mathbf{f}(x)=\left \{y\in Y: (x,y)\in \mathbf{f} \right\} $$ is nonempty, for every $x\in X.$  We often write $x\xrightarrow{\mathbf{f}} y$ whenever $(x,y)\in \mathbf{f},$ in which case $y$ is called an \emph{image} of $x,$ and $x$ a \emph{pre-image} of $y$ under $\mathbf{f}.$ So $\mathbf{f}(x)$ is the set of images of $x$ and  $\mathbf{f}^{-1}(y)$ is the set of pre-images of $y.$ More generally, \begin{equation}\mathbf{f}(A) =\bigcup_{x\in A} \mathbf{f}(x), \ \ \mathbf{f}^{-1}(B)=\bigcup_{y\in B} \mathbf{f}^{-1}(y),\end{equation} for every $A\subset X$ and $B\subset Y.$ 
 
\nid The composition $\mathbf{f}\circ\mathbf{g}:X\to Y$ of multifunctions $\mathbf{g}:X \to Z$ and $\mathbf{f}: Z \to Y$ is defined by the relation: $x\mapsto y$ iff there is $z\in Z$ such that  $x\xrightarrow{\mathbf{g}} z \xrightarrow{\mathbf{f}} y. $ Notice that $\textbf{log}(z)$ (the inverse of $\mathop{e}^{z}$) and $\textbf{arg}(z)$ are examples of multifunctions.

The restriction of a multifunction $\mathbf{f}:X \to Y$ to a subset $A\subset X$ is  $\mathbf{f}|_{A}=\mathbf{f}\cap A\times Y. $

 \nid \emph{We shall use boldface letters to denote multifunctions which are not single valued.}  The terms map, function, transformation, etc, are always meant to be single-valued.  
 
 If $f: U_0\to V_0$ and $g: U_1 \to V_1$ are functions,  then  by convention the domain of $g\circ f$ is the union of all sets $A\subset U_0$ such that $f(A) \subset U_1.$

\section{Holomorphic motions of hyperbolic sets}\label{first}

This paper is concerned with the one parameter family of holomorphic correspondences \begin{equation}\mathbf{f}_c:\hat{\mathbb{C}} \to \hat{\mathbb{C}}\end{equation}determined by integers $p>q\geq 2:$ \begin{equation}\label{HC}\mathbf{f}_c(z)=\left\{ w\in \mathbb{C}: (w-c)^q=z^{p}\right\}.\end{equation}

\subsection{Julia sets}\label{obsec}  Let $\mathbf{f}_c$ be the correspondence defined in \eqref{HC}. An  \emph{orbit} of $\mathbf{f}_c$ is any sequence $\{z_k\}$  of iterates \begin{equation} z_0 \xrightarrow{\mathbf{f}_c} z_1 \xrightarrow{\mathbf{f}_c}  z_2 \xrightarrow{\mathbf{f}_c} \cdots\end{equation}   Finite orbits $\{z_i\}_{0}^{n}$ for which $z_n=z_0$ are called \emph{cycles}. Every element of the cycle is a \emph{periodic point} of $\mathbf{f}_c.$

\nid A \emph{branch} of $\mathbf{f}_c$ is a holomorphic function $\varphi:U\to \mathbb{C}$ such that $ z\xrightarrow{\mathbf{f}_c} \varphi(z)$ on $U.$  Provided no point of the cycle $\{z_k\}$ belongs to $\{0,\infty\},$ we may apply the implicit function theorem to determine branches $\varphi_k$ of  $\mathbf{f}_c$ taking $z_{k-1}$ to $z_k.$ Each branch $\varphi_k$ is univalent (i.e., injective, which implies that $\varphi_k$ defines  a bi-holomorphic map onto its image).  Let $\Phi=\varphi_n\circ \cdots \circ \varphi_1.$ \emph{Repelling cycles} are those for which \begin{equation}|\Phi'(z_0)|>1.\end{equation}   If some point of the cycle belongs to $\{0,\infty\},$ or $|\Phi'(z_0)| <1,$ then  $\{z_k\}$ is an \emph{attracting cycle} of $\mathbf{f}_c.$
\emph{The Julia set} $J(\mathbf{f}_c)$ is the closure of the union of all repelling cycles of $\mathbf{f}_c;$ the \emph{dual Julia set} $J^*(\mathbf{f}_c)$ is the closure of the union of all attracting cycles of $\mathbf{f}_c.$ We often write $J_c=J(\mathbf{f}_c)$ and  $J_c^*=J^*(\mathbf{f}_c).$ It can be easily checked that for every correspondence $(w-c)^q=z^{p}$ there is a  positive radius $R=R(c,p,q)$ such that $D_R=\{|z| > R\}$ is fully invariant under $\mathbf{f}_c,$ i.e., if $z\xrightarrow{\mathbf{f}_c} w$ and $|z|> R,$ then $|w| > R.$ In this case, every orbit in $D_R$ converges exponentially fast to $\infty.$ This implies $J_c\subset D_R$ and $J_c^* \subset D_R.$ 

\subsection{Hyperbolic sets}\label{sect25}   Suppose $\Lambda\subset \mathbb{C}^*$ is a compact set which is {\it semi-invariant} in the following sense: every point $z\in \Lambda$ has at least one image $w\in \Lambda$ and at least one pre-image $\zeta\in \Lambda$ under $\mathbf{f}_c.$  This property implies $c\not \in \Lambda,$ and so $\Lambda$ is contained in some $A_r \cap (c+ A_r),$ where \begin{equation}A_r=\left\{z\in \mathbb{C}: r< |z| < \frac{1}{2r}\right\}.\end{equation}  The class of  semi-invariant compact sets in $A_r\cap (c+ A_r)$  is denoted by $\inv(A_r,c).$  We say that $\Lambda\in \inv(A_r,c)$ is a \emph{hyperbolic set} of $\mathbf{f}_c$ if there is a conformal metric $\mu(z)|dz|$ defined on a region $V$ containing  $\Lambda$ such that  either 
\begin{equation}\label{nice}
(a) \   \ \inf_{z, \varphi} \|\varphi'(z)\|_{\mu} > 1  \ \ \ \ \ \ \textrm{or}  \ \ \ \ \ \ \ \   (b) \ \ \sup_{z,\varphi} \| \varphi'(z)\|_{\mu}<1,
\end{equation}
where $\sup$ and $\inf$ are taken over all $z\in V$ and all possible branches $\varphi$ of $\mathbf{f}_c$ defined at $z$ such that $\varphi(z)\in V.$

 The class of all hyperbolic sets $\Lambda$ in  $ \inv(A_r,c)$ is denoted by $\hyp,$ which can be partitioned  into two disjoint {\it hyperbolic classes}: $\Lambda$ is \emph{expanding}  if it satisfies $(a)$ of  \eqref{nice};  it is attracting if  $\Lambda$ satisfies  condition $(b).$
   
An expanding hyperbolic set is a \emph{hyperbolic repeller} if  $\mathbf{f}_{c}^{-1}(\Lambda)=\Lambda.$  If $\Lambda$ is an attracting hyperbolic set such that $\mathbf{f}_c(\Lambda)=\Lambda,$  then $\Lambda$ is a  \emph{hyperbolic attractor} of $\mathbf{f}_c.$  
 We say that $\Lambda=J(\mathbf{f}_{c})$ is  a \emph{hyperbolic Julia set} if $\Lambda$ is hyperbolic and expanding.  The  dual Julia set $J_c^*$ of $\mathbf{f}_{c}$ is \emph{hyperbolic} if $J_c^*$ is hyperbolic and attracting.

\begin{lem}\label{naive} Let $c\in \mathbb{C}$ and suppose $\displaystyle0< r < \min\left \{{1/\sqrt{2}}, \  (4|c|)^{-1}  \right \}.$ There is $\delta>0$ such that for every $\Lambda \in \inv(A_r,c)$ \emph{and} $z\in \Lambda,$  each of the sets  \begin{equation}\mathcal{D}_1^{+}(z)=r\mathbf{f}_c(z) \cap \Lambda + \delta \overline{\mathbb{D}},\end{equation} \begin{equation}\mathcal{D}_1^{-}(z)=r\mathbf{f}^{-1}_c(z) \cap \Lambda + \delta \overline{\mathbb{D}},\end{equation}  consists of finitely many disjoint closed disks of radius $\delta$ contained in $\cdisk.$ 
\end{lem}

\begin{proof} For $\mathcal{D}_1^{-}(z)$ the proof proceeds as follows. Let $z\in \Lambda.$ The points of $\mathbf{f}_c^{-1}(z)$ are the solutions of $\zeta^{p}=(z-c)^{q},$ and are symmetrically disposed on a circle of radius greater than $r^{q/p}$ centred at $0.$ Since \begin{equation}r(\mathbf{f}_c^{-1}(z) \cap \Lambda) \subset r(A_r +c) \subset D(3/4):=\left\{z\in \mathbb{C}: |z| < 3/4 \right\},  \end{equation} it follows that $r(\mathbf{f}_c^{-1}(z) \cap \Lambda)$ is a finite set of points $a_i$ in $D(3/4)$ with $|a_i - a_j|> \rho_r,$ where the constant $\rho_r$ depends only on $r.$ Keeping $r$ fixed, for $\delta< 1/4$ sufficiently small we conclude that $r(\mathbf{f}_c^{-1}(z) \cap \Lambda) + \delta \overline{\mathbb{D}}$  consists of finitely many disjoint closed disks contained in $\overline{\mathbb{D}}.$ A similar and slightly simpler proof can be given for  $\mathcal{D}_1^{+}(z).$\end{proof}

\subsection{Cantor bundles}\label{bundles} Let $\Lambda \in \inv(A_r,c).$ For a given $z\in \Lambda,$ let  $\mathcal{D}_{n}^{+}(z)$ be the union of all disks \begin{equation}\label{form}r(z_1+ \delta z_2 +\cdots + \delta^{n-1}z_n) + \delta^n\overline{\mathbb{D}}\end{equation} such that \begin{equation}\label{series}z\xrightarrow{\mathbf{f}_c} z_1 \xrightarrow{\mathbf{f}_c} z_2 \xrightarrow{\mathbf{f}_c} \cdots, \ \  z_i\in \Lambda. \end{equation}
Using Lemma \ref{naive},  we see that $\mathcal{D}_{n+1}^{+}(z)$ can be obtained from $\mathcal{D}_{n}^{+}(z)$ after replacing $\overline{\mathbb{D}}$ in each disk \eqref{form} by the corresponding set of smaller disks $\mathcal{D}_{1}^{+}(z_n) \subset \mathbb{D}.$  As a result,  we have a nested sequence of compact sets $\mathcal{D}_{n+1}^{+}(z) \subset \mathcal{D}_{n}^{+}(z).$ Therefore \begin{equation}\mathcal{C}^{+}(z)=\bigcap_{n>0} \mathcal{D}_{n}^{+}(z) \end{equation} is nonempty and every point of $\mathcal{C}^{+}(z)$ is uniquely expressed as 
\begin{equation}\label{north} s=r\sum_{n=1}^{\infty} \delta^{n-1}z_n, \end{equation}
where $\{z_n\}$ is a sequence satisfying \eqref{series}. By taking backward iterates, we define $\mathcal{D}_{n}^{-}(z)$ as the union of all disks in the form \eqref{form} such that\begin{equation}\label{seriez}\cdots\xrightarrow{\mathbf{f}_c} z_2 \xrightarrow{\mathbf{f}_c} z_1 \xrightarrow{\mathbf{f}_c} z, \ \  z_i\in \Lambda. \end{equation} The intersection of all $\mathcal{D}_n^{-}(z)$ gives a set $\mathcal{C}^-(z)$ whose elements are uniquely written as a series of the form \eqref{north}, where $\{z_n\}$ must satisfy \eqref{seriez}.
Notice that $\mathcal{C}^{+}(z)$  is Cantor set if every point of $\Lambda$ has at least two images in $\Lambda.$ The same holds for $\mathcal{C}^{-}(z)$ with respect to backward iterates. Let 
\begin{equation} \label{level}\mathcal{B}^{+}_{r,\delta}(\Lambda) = \bigcup_{z\in \Lambda} \{z\} \times \mathcal{C}^{+}(z), \ \  \mathcal{B}^{-}_{r,\delta}(\Lambda) = \bigcup_{z\in \Lambda} \{z\} \times \mathcal{C}^{-}(z).\end{equation}

For any subset $K$  of $\Lambda,$  $\mathcal{B}_{r,\delta}^{+}(K)$  can be defined as the right hand side of the first equation in \eqref{level}, with $K$ in  place of $\Lambda.$ 
 The  \emph{Cantor bundle}  $\mathcal{B}(\Lambda)$ of a hyperbolic set is either $\mathcal{B}_{r,\delta}^{+}(\Lambda)$ when $\Lambda$ is expanding or $\mathcal{B}_{r,\delta}^{-}(\Lambda)$ when $\Lambda$ is attracting.  The reason why we have omitted the index $(r, \delta)$ is explained by Lemma \ref{thv}.

\nid Hence every point  $(z, w)\in \mathcal{B}(\Lambda)$ is uniquely given by $w=s,$ where $s$ is series  like \eqref{north} whose first term is $rz$, and the points $z_n$ must satisfy either \eqref{series} or \eqref{seriez}. The surjective map $\pi:\mathcal{B}(\Lambda) \to \Lambda$ given by $\pi(z, s)= z$ is the \emph{bundle projection}. 

\begin{lem}\label{thv}{\it Let $\Lambda \in \inv(A_r,c)$ and $\ep \in \{1, -1\}.$  Any two Cantor bundles $\mathcal{B}^{\ep}_{r,\delta}(\Lambda)$ and $\mathcal{B}^{\ep}_{r',\delta'}(\Lambda)$ are related by a homeomorphism $ \Phi: \mathcal{B}^{\ep}_{r,\delta}(\Lambda) \to \mathcal{B}^{\ep}_{r',\delta'}(\Lambda)$ such that  \begin{equation}\label{notice} \pi\circ \Phi(z,w)=z, \ \ (z,w)\in \mathcal{B}_{r,\delta}(\Lambda), \end{equation}
where $\pi:\mathcal{B}^{\ep}_{r', \delta}(\Lambda)\to \Lambda$ is the bundle projection.}\end{lem}

\begin{proof} Suppose $\ep=1.$ For every $(z,s) \in \mathcal{B}_{r,\delta}(\Lambda)$ there is a unique orbit  $z_i\in \Lambda_c$ of $\mathbf{f}_c$ starting at $z$ such that $s=r(z_1 + \delta z_2 + \cdots).$ The homeomorphism \begin{equation}\Phi: \mathcal{B}_{r,\delta}(\Lambda) \to \mathcal{B}_{r', \delta'}(\Lambda)\end{equation} is defined by $\Phi(z,s)=(z, s'),$ where $s'=r'(z_1 + \delta'z_2 + \cdots).$ The proof for $\ep=-1$ is similar.
\end{proof}

The function $f_c: U \to \mathbb{C}^2$ below is the bundle map associated to $\mathcal{B}(\Lambda).$ 
\begin{lem}[Bundle map] \label{semiconjugacy} {\it Suppose $\Lambda$ is a hyperbolic set of $\mathbf{f}_c.$ Then $\mathcal{B}(\Lambda)$ is compact and there is a  holomorphic  map $f_c: U \to \mathbb{C}^2,$ where $U$ is an open set containing $\mathcal{B}(\Lambda)$ and $f_c$ is given by \eqref{fgc} such that \begin{itemize} \item[$(a)$] $f_c: \mathcal{B}(\Lambda) \to \mathcal{B}(\Lambda)$ is a bijection; \item[$(b)$] $\pi: \mathcal{B}(\Lambda) \to \Lambda$ is a semi-conjugacy, in the sense that $\pi$ is surjective and every orbit  of $f_c$ in $\mathcal{B}(\Lambda)$  projects to  an orbit of $\mathbf{f}_c^{\ep}$ in $\Lambda,$ where $\ep=1$ if $\Lambda$ is expanding or $\ep=-1$ if $\Lambda$ is contracting;
\item[$(c)$] Let $J(f_c)$ be the closure of periodic points of $f_c: U \to \mathbb{C}^2.$ If  $\Lambda = J_c$ or $\Lambda=J_c^*$, then  $f_c$ restricts to a surjective map $f_c: J(f_c) \to J(f_c),$  and  \begin{equation}\pi J(f_c) = \Lambda;\end{equation} 
\item[$(d)$] for every $x\in \mathcal{B}_{r,\delta}(\Lambda),$ \begin{equation} x= \left (  \pi(x), \ r\sum_{n=1}^{\infty} \delta^{n-1} \pi f_c^n(x)\right).  \end{equation}
\end{itemize}
}
\end{lem}
\begin{proof}
For the construction of $f_c$ we may assume $\Lambda$ is expanding (the contracting case is similar and is obtained by taking branches of $\mathbf{f}_c^{-1}$ instead of $\mathbf{f}_c).$ There is an open set $O \supset \Lambda$ and $\ep>0$ such that (i) if $z\xrightarrow{\mathbf{f}_c} z_1,$ $ x \xrightarrow{\mathbf{f}_c} x_1$ and these points are in $O$ with $|z-x|< 2\ep$ and $|x_1-z_1| < 2\ep,$ then they are related by the same branch $\varphi$ of $\mathbf{f}_c,$ that is, $\varphi(z)=z_1$ and $\varphi(x)=x_1;$ and (ii) there is $\delta_{0} >0$ such that, for any $(z,w)$ and $(x,y)$ in $\mathcal{B}(\Lambda),$ 
\begin{equation}\label{plo} |(z,w) - (x,y)| < \delta_0 \implies  |x-z| < \ep, \  |x_1 - z_1| < \ep,\end{equation} where $w=r\sum_{1}^{\infty} \delta^{n-1} z_n$ and $y=r\sum_{1}^{\infty} \delta^{n-1}x_n.$  Let $U$ denote the $\delta_0$ neighbourhood of $\mathcal{B}(\Lambda)$ in $\mathbb{C}^2$ (the points within distance $\delta_0$ from $\mathcal{B}(\Lambda)). $ The \emph{bundle map} $f_{c}: U \to \mathbb{C}^2$ is defined by \begin{equation} \label{fgc}f_c(z,w) = (\varphi(z),  \ \delta^{-1} w  - r\delta^{-1} \varphi(z)),\end{equation} where $\varphi$ is a branch of $\mathbf{f}_c$ defined in the following way. Since   $(z,w)$ is $\delta_0$-close to a point $(x,y)\in \mathcal{B} (\Lambda),$ with $y=r\sum \delta^{n-1}x_n,$ we let  $\varphi$ be the branch of $\mathbf{f}_c$ that takes $x$ to $x_1.$  There is no loss of generality in assuming that  the domain of $\varphi$ includes the ball of radius $4\ep$ centred at $x.$ By \eqref{plo},  $z$ is in the domain of  $\varphi.$   The choice of $\varphi$ does not depend on $(x,y),$ for if $(\tilde{x}, \tilde{y})$ is another choice with $\tilde{y}= r\sum \delta^{n-1} \tilde{x}_n,$  then by \eqref{plo} the  finite orbits $x \xrightarrow{\mathbf{f}_c} x_1$ and $\tilde{x} \xrightarrow{\mathbf{f}_c} \tilde{x}_1$ are within distance $2\ep$ from each other, and by (i) they are  related by the same branch $\varphi.$ 
The  compactness of $\mathcal{B}(\Lambda)$ can be proved using the topological embedding $\phi: \mathcal{B}(\Lambda) \to \Lambda^{\mathbb{N}_0}$ given by \begin{equation}\phi(x)=\left ( \pi f_c^{n}(x)\right)_{n=0}^{\infty}.\end{equation} Since the image of $\phi$ is a closed subset of the product space (which is compact), we conclude that $\mathcal{B}(\Lambda)$ is compact. 
Every point $x\in \mathcal{B}(\Lambda)$ is presented  in the form $x=(z,w)$ where $w=r\sum_{1}^{\infty} \delta^{n-1}z_n,$ for some orbit $\{z_n\}$  of $\mathbf{f}_c$ starting at $z_0=z.$ An easy computation yields  
\begin{equation}\label{fde} f_c(z,w)= (z_1, r(z_2 + \delta z_3 + \delta^2 z_4 + \cdots)) \in \mathcal{B}(\Lambda).  \end{equation} Using this expression and the semi-invariance of $\Lambda,$  we prove that $f_c:\mathcal{B}(\Lambda) \to \mathcal{B}(\Lambda)$ is surjective and $\pi$ maps orbits to orbits (a semi-conjugacy).    If $Q$ is the set of periodic orbits of $f_c,$ then $\overline{Q}= J(f_c)$ and $f_c(Q)=Q.$ By continuity, we conclude $(c).$  Using \eqref{fde} we obtain $z_n=\pi f_c^n(x)$  and  therefore $(d).$     
\end{proof}

\subsection{Remark} Since $f_c:U \to \mathbb{C}^2$ does not map $U$ into itself, a periodic point of $f_c$ is obtained from a cycle of $f_c$ whose points, by definition, are contained in $U.$  In this paper $J(f_c)$ is the \emph{Julia set of $f_c,$} for the following reason. By hyperbolicity, every cycle of $\mathbf{f}_c^{\ep}$ in a small neighbourhood of $\Lambda$   is  repelling. Thus, for every periodic point $x= (z,w) \in U$ of period $n,$  
\begin{equation}\label{kcm}\det \operatorname{Jac}(f_c^n)(x)= \delta^{-n} |\varphi'(z)| > 1,\end{equation} where $\varphi$ is a branch of $(\mathbf{f}_c^{\ep})^n$ such that $\varphi(z)=z.$  Notice that $\varphi'(z)$ is the multiplier of the projected cycle. $\operatorname{Jac}$ stands for (complex) Jacobian matrix. We might have used \eqref{kcm} to define the  multiplier of any cycle of $f_c.$ In that sense, $J(f_c)$ would be simply the closure of repelling periodic orbits, even when $\Lambda=J_c^*$ is contracting. (Since $J_c^*$ is the projection of $J(f_c),$ this justifies the name of $J_c^{*}$  as the dual Julia set, but the standard definition(s) of Julia set for general maps of $\mathbb{C}^2$ is slightly different.) 

\subsection{Evenly covered sets}  
Suppose $\Lambda$ is a hyperbolic set of $\mathbf{f}_c$ and $L\subset \Lambda.$  We denote by $\embd(L)$ the space of all continuous maps  $\varphi:L \to \mathcal{B}(\Lambda)$ such that $\pi \varphi(z)=z$ on $L.$ Give $\embd(L)$ the uniform topology ($\sup$ metric).    We say that a compact set $L\subset \Lambda$ is \emph{evenly covered} if the corresponding bundle can be written as a disjoint union: \begin{equation}\mathcal{B}(L)=\bigcup_{\varphi\in \embd(L)}\varphi(L). \end{equation}

\begin{thmm}[Local product structure] \label{nossa}{\it Let $\Lambda$ be a hyperbolic set of $\mathbf{f}_c.$ If $L\subset \Lambda$ is evenly covered, then there is a homeomorphism \begin{equation}\Phi:  L\times \embd(L) \to \mathcal{B}(L).\end{equation}
In addition,  if every point of $\Lambda$ has at least two images in $\Lambda$ and at least two pre-images in $\Lambda$ under $\mathbf{f}_c,$ then $\embd(L)$ is a Cantor set.}
\end{thmm}
\begin{proof} The map $\Phi: L\times \embd(L)\to  \mathcal{B}(L)$ given by  $(z,\varphi) \mapsto \varphi(z)$ is continuous and surjective. It is also injective, for if $\Phi(z,\phi)=\Phi(w,\psi)$ then 
\begin{equation}z=\pi \varphi(z) = \pi \psi(w)  = w. \end{equation} Hence  $\varphi(z)=\psi(z).$ Since $L$ is evenly covered, we have $\varphi=\psi.$ Since $L$ is compact, $\Phi$ is a homeomorphism. 
\end{proof}

{\normalfont

\subsection{Holomorphic motions}\label{asc} Let $A\subset \mathbb{C}^n$ and suppose $U\subset \mathbb{C}$ is open and connected. We say that $h:U\times A \to \mathbb{C}^n$ is a \emph{holomorphic motion} if there is $c_0\in U$ -- called the \emph{base point of $h$} -- such that \begin{itemize}\item[$(a)$] $h(c_0,\cdot)$ is the identity; \item[$(b)$] each $h(c,\cdot):A \to \mathbb{C}^n$ is injective, and  \item[$(c)$] $h(\cdot,z):U\to \mathbb{C}^n$ is holomorphic, for every $z\in A.$
 It is very common to denote $h_c=h(c,\cdot).$  
\end{itemize}
If $O$ is an open set containing the hyperbolic set $\Lambda$ of $\mathbf{f}_c^{\ep},$ then $\operatorname{cycles}(O,c)$  denotes the set of cycles of $\mathbf{f}_c^{\ep}$ which are contained in $O,$ where either $\ep =1$ if $\Lambda$ is expanding or $\ep=-1$ if $\Lambda$ is attracting. 
 \begin{mthmm}[Holomorphic motions]\label{giocco}{\it Let $\Lambda_{c_0}$ be a hyperbolic set of $\mathbf{f}_{c_0}.$  There is a unique holomorphic motion \begin{equation}h: U\times \mathcal{B}(\Lambda_{c_0}) \to \mathbb{C}^2\end{equation} based at $c_0$ such that }{\it
\begin{itemize}
 \item[$(a)$] for every $c\in U,$ the projection $\Lambda_c=\pi h_c (\mathcal{B}(\Lambda_{c_0}))$ is a hyperbolic set of $\mathbf{f}_c$ varying continuously with respect to $c \in U$ in the Hausdorff metric. 
 
 \item[(b)] for any $c\in U,$ $\Lambda_c$  is attracting $($resp. expanding$)$ iff $\Lambda_{c_0}$ is attracting $($resp. expanding$)$;
 \item[$(c)$]  the image of $h_c$  is a subset of $\mathcal{B}(\Lambda_c)$ which is invariant under  $f_c:\mathcal{B}(\Lambda_c) \to \mathcal{B}(\Lambda_c).$ We also have  $h_c\circ f_{c_0}=f_c\circ h_c,$ for every $c\in U;$

  \item[$(d)$] there is a constant $C_0>0$ such that   \begin{equation}|h_u(x) - h_{v}(x)| \leq C_0|u-v|, \end{equation} for every $u,v\in U$ and $x\in \mathcal{B}(\Lambda_{c_0});$
  \item[$(e)$] for every sufficiently small open set $O$ containing the union of all $\Lambda_c, \ c\in U,$ we have  \begin{equation}\Lambda_c \supset\operatorname{cycles}(O,c) \iff \Lambda_{c_0} \supset \operatorname{cycles}(O,c_0),\end{equation}  for every $c\in U;$ and
  
   \item[$(f)$]{if $c\in U,$ substituting $c$ for $c_0$  yields  a unique holomorphic motion }\begin{equation} g: U\times \mathcal{B}(\Lambda_c) \to \mathbb{C}^2 \end{equation} parameterized on the same $U$, based at $c,$  satisfying  $(a)$-$(d)$ and the inverse property: if $x\in \mathcal{B}(\Lambda_c)$ and $y=g_{c_0}(x)$ is in $\mathcal{B}(\Lambda_{c_0}),$ then $h_{c}(y)=x.$   \end{itemize}}

 \end{mthmm}
\begin{proof} See \S  \ref{xxx}. \end{proof}

\section{Applications of Theorem \ref{giocco} } \label{this}

In this section we show how holomorphic motions of the Cantor bundle  $\mathcal{B}(\Lambda)$  can projected on $\mathbb{C}$  to obtain branched holomorphic motions of $\Lambda.$ This branched holomorphic motion {satisfies} a very important quasiconformal property (Theorem \ref{ofd}) which is better described with the {concept of a shadow.}

Let $\Lambda$ be a hyperbolic set of $\mathbf{f}_c$ and let $\pi$ be the bundle projection. For a given $\theta\in (0, 2\pi),$ let  $\operatorname{sh}(\Lambda, \theta)$ denote {the set of all} $L\subset \Lambda$ such that $(a)$ $L=\pi(K)$ for some connected and compact set $K\subset \mathcal{B}(\Lambda);$ and $(b)$ there is $\theta_0$ such that \begin{equation} \label{puff}\theta_0 < \arg(z)<\theta_0 + \theta, \ \ z\in L,\end{equation} for some branch $\arg(z)$ of the multifunction ${\textbf{arg}}(z).$ By \eqref{notice}, this definition is independent of the particular $(r,\delta)$ that defines $\mathcal{B}(\Lambda).$ {Let $\operatorname{sh}(\Lambda)$  consist of all $L$ such that $L \in \sh(\Lambda, \theta),$ for some $\theta\in (0,2\pi).$ } The elements of $\operatorname{sh}(\Lambda)$ are the \emph{connected shadows} of $\Lambda.$

\begin{thmm}[Characterization of shadows]  \label{ghf}{\it Let  $\Lambda$ be a hyperbolic set of $\mathbf{f}_c.$ If $K\subset \mathcal{B}(\Lambda)$ is connected and projects onto a shadow $L\in \sh(\Lambda),$ then \begin{equation}\pi:K\to L\end{equation} is a homeomorphism. Hence the shadows of $\Lambda$ are precisely those connected sets $L$ contained in a sectorial piece of $\Lambda$ -- defined by \eqref{puff} -- for which $\embd(L)$ is nonempty.}
\end{thmm}

\begin{proof}  Let $\Omega$ be the unbounded connected component of $\hat{\mathbb{C}} - L.$  Since $\hat{\mathbb{C}}-\Omega$ is the union of $L$ with all bounded components of $\mathbb{C} -L,$ it follows that the complement of $\Omega$ is compact and connected. Hence $\Omega$ is simply connected. Thus $\Omega$ is conformally isomorphic to the open unit disk $\mathbb{D}.$  Since $L\subset \Lambda,$ $0\not \in L.$ Since $L$ is compact and sectorial, it follows that $0\in \Omega.$  Using the Riemann mapping $\mathbb{D} \to \Omega$ we conclude that there is a simply connected set $U_1$ containing $L$ such that $0\not \in U_1.$ Let $g_n: K \to \Lambda$ be defined by $g_n(x)=\pi f^n(x),$ where $f:\mathcal{B}(\Lambda) \to \mathcal{B}(\Lambda)$ is the bundle map. By Lemma \ref{semiconjugacy} $(b)$ we have \begin{equation} g_n(x) \xrightarrow{\mathbf{f}_c^{\ep}} g_{n+1}(x),\end{equation}

\nid for every $x\in \mathcal{B}(\Lambda).$ There is no loss of generality in assuming  $\Lambda$ is expanding and $\ep=1.$ Now fix $x_0\in \mathcal{B}(\Lambda).$  Recall that if $\varphi: U\to \mathbb{C}$ is a holomorphic map defined on a simply connected domain $U$ with $0\not \in \varphi(U),$ then for every $z_0\in U$ and every  $w_0 \in \textbf{log}\, \varphi(z_0),$
there is a unique holomorphic
 branch $\varphi_1$ of $\textbf{log}\,\varphi(z)$ defined on $U$ such that $\varphi_1(z_0)=w_0.$ Hence there is a unique holomorphic branch $\varphi_1: U_1 \to \mathbb{C}$ of \begin{equation}\mathbf{f}_c(z)= \exp \frac{1}{q}\textbf{log}\,z^{p} + c\end{equation} such that $\varphi_1$ sends $g_0(x_0)$ to $g_1(x_0).$ The set $\{x\in K: \varphi_1\,(g_0(x)) = g_1(x)\}$ is therefore nonempty and closed. It is also open, as can be easily checked from the fact that there is $d>0$ such that $|w - w'|>d,$ for every $z\in \Lambda$ and every  $w, w' \in \mathbf{f}_c(z)$ with $w \neq w'.$ We conclude that $\varphi_1(g_0(x)) = g_1(x),$ for every $x\in K.$ Since $\varphi_1(L) \subset \Lambda,$ there is a smaller simply connected $U_2\subset U_1$  such that $L\subset U_2$ and  the closure of $\varphi_1(U_2)$ does not contain $0.$ There  is a holomorphic branch $\varphi_2$ of $\mathbf{f}_c(\varphi_1(z))$ defined on $U_2$ such that $\varphi_2$ sends $g_0(x_0)$ to $g_2(x_0).$ Again, the set of points $x\in K$ for which  $\varphi_2(g_0(x)) = g_2(x)$ is nonempty, open and closed in $K.$ Hence $\varphi_2 \circ g_0 = g_2$ on $K.$ The process can be repeated indefinitely so as to obtain a sequence of holomorphic maps $\varphi_n: U_n \to \mathbb{C} $
 such that $L\subset U_{n+1} \subset U_n, \ \varphi_n(z)  \xrightarrow{\mathbf{f}_c} \varphi_{n+1}(z),$ and  $\varphi_n \circ g_0(x) = g_n(x),$ for every  $z\in L$ and $x\in K.$ Let $\Phi: L \to K$ be defined by 
 \begin{equation}\Phi(z)= \left (z, r\sum_{n=1}^{\infty} \varphi_n(z) \delta^{n-1}\right). \end{equation} By Lemma \ref{semiconjugacy} $(d),$ \  $\Phi \circ \pi (x)= x,$ for $x\in K.$ Hence $\Phi$ is a surjective continuous map such that   $\pi \Phi(z) = z$ on $L.$ Therefore, $\Phi: L \to K$ is a homeomorphism. 
\end{proof}

\subsection{Branched holomorphic motions}\label{pld} Let $A\subset \hat{\mathbb{C}}$ and suppose  $U\subset \mathbb{C}$ is open and connected. Let $\mathbf{h}: U\times A \to \hat{\mathbb{C}}$ be a multifunction. For every $c\in U,$ let $\mathbf{h}_c:A \to \hat{\mathbb{C}}$ be the multifunction given by $\mathbf{h}_c(z)=\mathbf{h}(c,z).$  We say that $\mathbf{h}$ is a \emph{branched holomorphic motion} if there is $c_0\in U$ -- called the base point of $\mathbf{h}$ --, such that 
 \begin{itemize}\item[$(a)$] $\mathbf{h}_{c_0}(z)=\{z\},$ for every $z\in A;$
 \item[$(b)$] If $G_z(\mathbf{h})$ is the union of all $\{c\} \times \mathbf{h}_c(z)$ such that $c\in U,$ then there is a family $\mathcal{F}$ of holomorphic functions $f: U\to \mathbb{C}$ such that 
\begin{equation}\label{inthat}\bigcup_{z\in \Lambda}G_z(\mathbf{h}) = \bigcup_{f\in \mathcal{F}} \mathcal{G}_{f}
\end{equation}
where $\mathcal{G}_f$ is the graph of the function $f.$
 \end{itemize} 
 
 \nid \emph{Normal branched holomorphic motions} are those for which $\mathcal{F}$ in \eqref{inthat} is a normal family.
 If $\mathbf{h}:U\times A \to \hat{\mathbb{C}}$ is a (normal) branched holomorphic motion, then so is every restriction $\mathbf{h}|_{U\times B},$ where $B\subset A.$   Branched holomorphic motions were originally defined by  \cite{Lyubich2015} for automorphisms of $\mathbb{C}^2.$
\subsection{Solenoidal extension} Let $\Lambda_{c_0}$ be a hyperbolic set of $\mathbf{f}_{c_0}$ and  let $S$ be a subset of  $\mathcal{B}(\Lambda_{c_0}).$ We say that \begin{equation}[S,\Lambda_{c_0}]\end{equation} is a \emph{solenoidal extension} if  $S$ is invariant: $f_{c_0}(S) \subset S, $ where   $f_{c_0}: \mathcal{B}(\Lambda_{c_0}) \to \mathcal{B}(\Lambda_{c_0})$ is the bundle map. A shadow $L \in \sh(\Lambda_{c_0})$ is induced by $S$ if $L=\pi(K)$ for some compact and connected set $K\subset S.$  Consider the multifunction $\mathbf{h}: U \times \Lambda_{c_0} \to \mathbb{C}$ given by 
\begin{equation}\label{qwe}\mathbf{h}(c,z)= \pi\circ h_c \circ \left(\pi|_S\right)^{-1}(z), \end{equation} where $U$ is a bounded domain and  $h: U \times \mathcal{B}(\Lambda_{c_0}) \to \mathbb{C}^{2}$ is the unique holomorphic motion based at $c_0$ satisfying $(a)$-$(d)$ of Theorem \ref{giocco}.  According to the following theorem, $\mathbf{h}$ is unique up to domain extensions, and so it makes sense to define $\mathbf{h}$ as   \emph{the branched holomorphic motion induced by the solenoidal extension} $[S, \Lambda_{c_0}].$ 

\begin{thmm}[Normal BHM]   \label{efd}{\it  The map $\mathbf{h}$ induced by $[S,\Lambda]$ is a normal branched holomorphic motion. Moreover, $\mathbf{h}$ does not depend on the particular choice of $(r,\delta)$ in the definition of $\mathcal{B}(\Lambda)$ and $\mathbf{h}$ is unique up to extensions of $U.$
}
\end{thmm}
\begin{proof} Let $\mathbf{h}: U \times \Lambda_{c_0} \to \mathbb{C}$ be induced by $\pi: S \to \Lambda_{c_0}.$   Then $\mathbf{h}_{c_0}(z)=\{z\},$ for every $z\in \Lambda_{c_0},$ where $\mathbf{h}$ is given by \eqref{qwe}. Consider the family $\mathcal{F}$ of all maps $f_{x,z}: U\to \mathbb{C}$ given by 
\begin{equation}f_{x,z}(c)=\pi h_c(x), \end{equation} where $\pi(x)=z\in \Lambda$ and $x\in S.$ By Theorem \ref{giocco} $(d),$ $\mathcal{F}$ is a bounded family of holomorphic functions on $U.$ Hence $\mathcal{F}$ is normal. We have \begin{equation}\bigcup_{z\in \Lambda}G_z(\mathbf{h})= \bigcup_{f\in \mathcal{F}} \mathcal{G}_{f}. \end{equation} It follows that every branched holomorphic motion  $\mathbf{h}$ induced by $[S, \Lambda_{c_0}]$ is normal.    By the uniqueness part of Theorem \ref{giocco}, there is only one  such $\mathbf{h}$ parameterized on $U.$   By Lemma \ref{thv}, $\mathbf{h}$ does not depend on the  $(r,\delta)$ that is used for $\mathcal{B}(\Lambda_{c_0}).$ \end{proof}

We say that a topological embedding $\varphi: L \to \mathbb{C}$ of a nonempty subset $L$ of the plane is a \emph{$K$-quasiconformal map} if

\begin{equation}
\sup_{z\in L} \limsup_{\ep \to 0} \frac{\sup_{|z-w|=\ep} |\varphi(z) - \varphi(w)|}{\inf_{|z-w|=\ep} |\varphi(z)-\varphi(w)|} \leq K.
\end{equation}
It is clear from Theorem \ref{giocco} that every branched holomorphic motion $\mathbf{h}: U \times \Lambda_{c_0} \to \mathbb{C}$ induced by a solenoidal extension of a hyperbolic set satisfies: for any $c\in U,$  $\mathbf{h}_c(\Lambda_{c_0}) = \Lambda_c$ is a hyperbolic set of $\mathbf{f}_c$ and  
\begin{equation}\disth(\Lambda_c, \Lambda_{c_0}) \leq C_0 |c-c_0|.\end{equation}

The quasiconformal properties of the branched holomorphic motion are given by 
 \begin{thmm}[Quasiconformality]   \label{ofd}{\it According to the notation of Theorem \ref{efd}, for every $\theta\in (0,2\pi)$ there is $r>0$ such that for every $$c\in D_{r}=\{|c-c_0| < r\} \subset U  \  \textrm{and}   \, \ L\in \sh(\Lambda_{c_0},\theta),$$
\begin{equation}\mathbf{h}_c(L)= \bigcup_{g\in \mathcal{H}} g_{c} (L)  \ \textrm {and}  \ g_c: L \to \mathbb{C} \  \textrm{is $K_{r}$-quasiconformal} \end{equation} where $K_r$ depends only on $r$ and   $\mathcal{H}$  is a family of holomorphic motions $g:D_{r}\times L \to \mathbb{C}$ such that $g_{c}(z) \in \mathbf{h}_c(z),$ for every $c\in D_r$ and $z\in L.$   }
 \end{thmm}

 \begin{proof}
  Apply Theorem \ref{giocco}.  Each holomorphic motion $g\in \mathcal{H}$ is given by $g(c,z)=\pi \circ h_c\circ \phi(z)$ for some $\phi \in \mathop{\mathrm{emb}}{(L)}.$ Theorem \ref{ghf}  is therefore used to prove the injectivity of the holomorphic motion. In order to prove  quasiconformality with uniform constant, fix arbitrary $z_0, z_1$ and $z_3$ in $L$ with $|z_1-z_2| = |z_3-z_2|=\ep.$ Consider the holomorphic function  between hyperbolic Riemann surfaces $f: U \to V$, where $V= \hat{\mathbb{C}} - \{0, 1, \infty\}$  and \begin{equation} f(c) = \frac{g_c(z_1) - g_c(z_2)}{ g_c(z_3) - g_c(z_2)}.  \end{equation} Since each $g_c$ is injective, the range of $f$ is indeed contained in $V.$ By the Schwarz Lemma,  $f$ does not increase the hyperbolic metric.  It is immediate that for every $r$ such that the closure of $D_r$ is contained in $U,$ there is  a constant $K_r$ depending only on $r$ such that $|f(c)| \leq K_r$ for $c\in D_r.$  This proves the $K_r$-quasiconformality.  \end{proof}
  
 We may apply Theorem \ref{giocco} for hyperbolic Julia sets $\Lambda_{c_0}=J_{c_0},$ obtaining hyperbolic sets $\Lambda_c$ close to $J_{c_0}.$    
Is it true that $J_c=\Lambda_c$?  The following theorem gives an answer.

\begin{thmm}[Branched motions of Julia sets]  \label{fds} {\it Suppose  $\Lambda_{c_0}=J_{c_0}$ is  a hyperbolic  set of $\mathbf{f}_{c_0}. $ Using the notation of Theorem \ref{giocco}, we have $J(f_c) \subset \mathcal{B}(\Lambda_c)$ and \begin{equation}h_c: J(f_{c_0}) \to J(f_c) \end{equation} is a homeomorphism, for every $c\in U.$  There is an open set $\Omega$ containing $J_{c_0}$ such that, for every $c\in U,$ \begin{equation}\mathbf{h}_c(J_{c_0}) = J_c \cap \Omega = \pi J(f_c),\end{equation} where  
 $\mathbf{h}: U \times J_{c_0} \to \mathbb{C} $ is the branched holomorphic motion induced by $[J(f_{c_0}), J_{c_0}].$ }\end{thmm}

 \begin{proof} Let $h: U\times \mathcal{B}(J_{c_0}) \to \mathbb{C}^2$ be the  holomorphic motion given by Theorem \ref{giocco} and let $\Lambda_c=\pi h_c(J_{c_0}).$  By Theorem  \ref{semiconjugacy} $(c),$ $J(f_c) \subset \mathcal{B}(\Lambda_c),$  for every $c\in U.$ It follows from Theorem  \ref{giocco} $(f)$ that  \begin{equation}h_c: J(f_{c_0}) \to J(f_c)\end{equation} is a bijection. Hence the branched holomorphic motion $\mathbf{h}$ induced by $[J(f_{c_0}), J_{c_0}]$ satisfy $\mathbf{h}_{c}(J_{c_0}) \subset J_c \cap \Omega.$ The other inclusion is obtained from Theorem \ref{giocco} $(e).$  \end{proof}

  \subsection{Remark} \label{dfg} A similar statement of Theorem \ref{fds} is valid for hyperbolic dual Julia sets.
  Notice that if  $J_c$ is connected then $\mathbf{h}_c(J_{c_0})= J_c \cap \Omega$ implies $\mathbf{h}_c(J_{c_0})=J_c.$ This stronger result also holds when $J_{c_0}$ satisfies $(\star):$ for every $\ep >0$ there is $\delta >0$ such that  \begin{equation}J_c \subset \{z\in \mathbb{C}: d_e(z, J_{c_0})< \ep \},\end{equation} whenever $d_e(c, c_0) < \delta,$ where $d_e$ denotes Euclidean distance.  (In this case, $J_c \subset \Omega.)$

\section{A detailed analysis of Julia sets for parameters close to zero }
\nid Theorems \ref{ofd}, \ref{fds} and Remark \ref{dfg} may be applied to the Julia sets $J_c$ for $c$ close to zero, since:

\begin{thmm}\label{hooi} {\it The Julia set $J_0 = \mathbb{S}^1$ is hyperbolic and satisfies $(\star)$ of Remark \ref{dfg}. }

\end{thmm}

\begin{proof}  First we prove that $J_{0}=\mathbb{S}^1.$  The iterates of $\mathbf{f}_0$ converge to infinity on the unbounded component of $\mathbb{C} - \mathbb{S}^1.$ On the other hand, the iterates of $\mathbf{f}_0$ converge to zero on  $\mathbb{D}.$ Hence every cycle of $\mathbf{f}_0$ which is neither $0\mapsto 0$ nor $\infty\mapsto \infty$ must be contained in $\mathbb{S}^{1}.$ The points of a cycle in $\mathbb{S}^1$  are given by solutions of the equation $z^{q^n}= z^{p^n},$ which become dense in $\mathbb{S}^1$ as $n$ increases. Since every cycle in $\mathbb{S}^1$ is repelling, it follows that $J_0=\mathbb{S}^1.$

Now we are going to prove that  $J_0$ satisfies $(\star).$  Let $0<r_c< R_c< 1$ be the  solutions of $g_c(x)=0$ on $(0, \infty),$ where \begin{equation}g_c(x)=x^{p/q} -x +|c|,\end{equation} for $|c|>0$ sufficiently small.  Since $g_c <0$ on the interval $(r_c, R_c),$ it can be shown that 

\nid {\bf (1)} whenever $z\xrightarrow{\mathbf{f}_c}w$ and $ r_c < |z| < R_c,$ we have $|w| < |z|;$  \\ {\bf (2)} there is also  $0<\rho < 1$ such that every cycle  in $B(\rho)=\{|z|<\rho\}$ is attracting. \\  If $s_c = (1+|c|)^{1/\gamma},$ where $\gamma=p/q-1,$ then whenever $z\not \in D(s_c)=\{|z|\leq s_c\}$ and $z\xrightarrow{\mathbf{f}_c}w,$ we have  \begin{equation}
|w| \geq |w-c| -|c| = |z|^{p/q} -|c| = |z|(|z|^{\gamma}-1) -|c| +|z| > |z|.
\end{equation}
As $|c|\to 0,$ we have  \begin{equation}0< r_c < \rho < R_c < 1 < s_c, \end{equation}
\begin{equation}r_c \to 0+, ~R_c \to 1-,~ s_c\to 1+,\end{equation} 
\begin{equation}\mathbf{f}_c(D(r_c)) \subset B(\rho).\end{equation}

\nid 
Now assume $z_0 \mapsto z_1 \mapsto \cdots \mapsto z_n=z_0$ is a cycle of $\mathbf{f}_c.$ By {\bf (1)} and {\bf (2)}, if one point of this cycle belongs to $B(R_c),$ then the whole cycle must be contained in $B(\rho).$ Hence, all cycles contained in $B(R_c)$ are attracting. If one of the points of a cycle, say $z_0,$ is in $\hat{\mathbb{C}}- D(s_c),$ then $z_i=\infty$ for all $i.$ We conclude that $J_c\subset D(s_c) - B(R_c).$  Property $(\star)$ of Remark \ref{dfg} follows when $c\to 0.$  \end{proof}

\subsection{Geometric construction of the solenoid $\mathcal{B}(\mathbb{S}^1)$} \label{ksc} It can be shown that $\mathcal{B}(J_0)$ is the limit set of a finite system of maps $u_k$ on the solid torus $\mathbb{T}.$
\nid For  every $k\in \mathbb{Z},$ let   $u_k:\storus \circlearrowleft$ be defined by 
\begin{equation} \label{hey} u_k(\mathrm{e}^{it}, z) = \left (\exp i\left({\frac{q}{p}t + \frac{2k\pi}{p}}\right),~\delta z + {r\mathrm{e}^{it}}\right ),  \ \ t\in [0, 2\pi), \  z\in \overline{\mathbb{D}}. \end{equation}

For any $A\subset \mathbb{T},$ let  \begin{equation}\omega(A)=\bigcup_{k=0}^{p} u_k(A).\end{equation}
It is possible to show that $\mathcal{B}_{r,\delta}(\mathbb{S}^1) $ is given as the nested intersection of all compact sets $\omega^n(\mathbb{T}).$ Each connected component of $\mathcal{B}(\mathbb{S}^1)$ in this case is a connected solenoid given as the inverse limit of $z\mapsto z^{q/d}$ on $\mathbb{S}^1,$ where $d=\gcd \{p,q\}.$ None of these facts is used in this paper, but this construction helps to visualize the solenoidal structure of $\mathcal{B}(J_0).$  When $q=2, p=1$ and $c=0,$ we have the correspondence $w^2=z,$ and the Cantor bundle $\mathcal{B}(\mathbb{S}^1)$ is the well-known Smale-Williams solenoid (the inverse limit of $z\mapsto z^2$ on $\mathbb{S}^1$).

  {\subsection{Symbolic construction of the solenoid. }}
 
 {Consider the maps $\theta_k: \mathbb{R} \to \mathbb{R}$ given by  $\theta_k(t) =pt/q + 2\pi k/q,$ for $0\leq k < q.$ }

  {An equivalent symbolic description of $\mathcal{B}(\mathbb{S}^1)$ is obtained from  $g:\mathbb{R}\times \Sigma_q \to \mathbb{C}^2, $ where $ \Sigma_q= \{0, \ldots, q-1 \}^{\mathbb{N}_0}$ and}
 
  $$g(t, \tau)= \left(e^{it}, r \sum_{n=1}^{\infty} \delta^{n-1} e^{i \theta_{k_n} \circ \cdots \circ \theta_{k_0} (t)} \right), $$ 
 where $\tau=(k_n)$ is any sequence in $\Sigma_q.$

The relation $\sim_g$ on $\mathbb{R} \times \Sigma_q$ is determined  by $x\sim_g y$ iff $g(x)=g(y).$
\begin{thmm} \label{int} The map $g: \mathbb{R} \times \Sigma_q \to \mathcal{B}(\mathbb{S}^1)$ is a continuous surjection which descends to a homeomorphism $\mathbb{R}\times \Sigma_q/\sim_g \to \mathcal{B}(\mathbb{S}^1).$ The restriction $g: [0, 2\pi)\times \Sigma_q \to \mathcal{B}(\mathbb{S}^1)$ is bijective.
\end{thmm}
 \begin{proof} The continuous and surjective map $g$ descends to a homeomorphism since the  quotient $\mathbb{R}\times \Sigma_q/ \sim_g$ can be also given by the quotient of a compact space, namely,  $[0, 2\pi] \times  \Sigma_q/ \sim_g.$ Using the definition of Cantor bundles and the property: {\it if $e^{it} \xrightarrow{\mathbf{f}_c}  e^{is}$ then there is only one $k$ modulo $q$ such that $s =\theta_k(t),$} we conclude that $g:[0, 2\pi) \times \Sigma_q \to \mathcal{B}(\mathbb{S}^1)$ is bijective.  
    \end{proof}
 \nid A similar construction gives  $\mathcal{B}(\mathbb{S}^1)$ as a quotient of $X=\{0, \ldots, q-1\}^{\mathbb{Z}}.$ Consider  $f: X \to \mathcal{B}(\mathbb{S}^1)$ given by $$f(\ldots, l_2, l_1, k_0, k_1, \ldots) =g(t, (k_0, k_1, \ldots)),$$
where $t= 2\pi \sum_{1}^{\infty} l_j/ q^{j}.$   The map $f$ is a continuous surjection  which descends to a homeomorphism  $X/\sim_f \to \mathcal{B}(\mathbb{S}^1).$ 
 
 {\subsection{Expanding maps}} \label{vcb}\nid We say that a covering map $f: (X,d) \to (X,d)$ of a compact metric space $X$ is \emph{expanding} if there is a constant $C>1$ such that  \begin{equation}\label{mns} d(f(x) , f(y)) > C \,d(x,y) \end{equation}
locally on $X.$ (There is a more general definition of distance expanding maps, but the above definition agrees with the standard terminology when $f$ is a covering map.) Since $X$ is compact,  there is $\delta>0$ such that the local branches of $f^{-1}$ on evenly covered sets of diameter $\delta$ contract by $1/C,$ and   \eqref{mns} holds on every ball of radius $\delta.$ 
 A continuous map $f: X \to X$ is said to be \emph{locally eventually onto} (shortly, a \textit{leo} map) if, for every  nonempty open subset $U$ of $X,$ there is $n>0$ such that $f^n(U)=X.$ (Hence every leo map is topologically mixing).  

It is well known that if $f: (X,d) \to (X,d)$ is an expanding leo map, then $X$ is the closure of the periodic points of $f.$ 

\subsection{The leo property for correspondences} \nid Given $\Lambda \in \inv(A_r,c),$ let $\mathbf{f}_{c,\Lambda}: \Lambda \to \Lambda$ be given by $\mathbf{f}_{c,\Lambda} (z)  =\mathbf{f}_{c}(z) \cap \Lambda.  $  We may define $\mathbf{f}_{c,\Lambda}^{-1}$ in a similar way.  
If $\Lambda$ is a hyperbolic repeller of $\mathbf{f}_c,$ then $\mathbf{f}_c$ is \textit{leo} on $\Lambda$ when every point of $\Lambda$ has a neighbourhood $O$ such that $\mathbf{f}^n_{c,\Lambda}(O \cap \Lambda)=\Lambda,$ for some $n>0.$ 
If $\Lambda$ is a hyperbolic attractor, then $\mathbf{f}_c$ is leo  on $\Lambda$ when  every point of $\Lambda$ has a neighbourhood $O$ such that $\mathbf{f}_{c,\Lambda}^{-n}(O \cap \Lambda)=\Lambda,$ for some $n>0.$

  \begin{thmm}[Mixing hyperbolic sets] \label{jsd}{\it Let $\Lambda$ be either a hyperbolic attractor or a hyperbolic repeller of $\mathbf{f}_c$ and suppose $\mathbf{f}_c$ is leo on $\Lambda.$  Then  $f_{c}: \mathcal{B}(\Lambda) \to \mathcal{B}(\Lambda)$ is  leo and expanding  for any metric in the family
\begin{equation}d_s(x,y)=\sum_{n=0}^{\infty} s^{n} \left |\pi f_c^{n}(x) - \pi f_c^n(y) \right|,  \end{equation}
where $s\in (0,1).$ In particular, \begin{equation}\mathcal{B}(\Lambda)= J(f_c); \end{equation}
$\Lambda \subset J_c$ if $\Lambda$ is a hyperbolic repeller; and $\Lambda \subset J_c^*$ if $\Lambda$ is a hyperbolic attractor. } \end{thmm}

\begin{proof} Suppose $\Lambda$ is a leo  hyperbolic repeller of $\mathbf{f}_c,$  and fix a metric $d_s.$ Let $X=\mathcal{B}(\Lambda)$ and $x\in X.$  There is an open set $U\subset  X$ containing $x$ such that $\mathbf{f}_c^{-1}(\pi(U))$ can be written as a disjoint union $\cup_{i} \varphi_i (\pi(U)),$ where $\varphi_i$ is a branch of $\mathbf{f}_c^{-1}$ defined on a neighbourhood of $\pi(x)$ on which \begin{equation} |\varphi_i(y) - \varphi_i(z)| < \lambda |y-z|,\end{equation} for some constant $\lambda < 1 $ that depends only on the hyperbolic set $\Lambda.$  Let $1/C=\lambda(1-s) +s.$ Then   $C>1,$ and by Lemma \ref{semiconjugacy} $(d),$ each  $g_i:U \to X$ given by \begin{equation}g_i(x)=\left ( \varphi_i(\pi(x)), r\sum_{n=0}^{\infty} \delta^n \pi f_c^n(x) \right ) \end{equation} is a branch of $f_c^{-1}.$  The backward invariance of hyperbolic repellers is used to prove that $g_i$ is well defined. Since the  sets $g_i(U)$ are disjoint and cover $f_c^{-1}(U),$ we conclude that $f_c$ is a covering map. If $x,y \in U,$ then  by  definition of $d_s,$ we have
\begin{equation}d_s(g_i(x), g_i(y))  < \lambda (d_s(x,y) -sd_s(x,y))+ s d_s(x,y) = \frac{1}{C} d_s(x,y). \end{equation} 

\nid Hence $f=f_c,$  $C$ satisfy  \eqref{mns} locally on $X,$ and $f_c$ is an expanding map. 

Let $U$ be any nonempty open subset of $X.$ By the definition of $d_s$ and the uniform expansion of the branches of $\mathbf{f}_c$ on $\Lambda$,  for any $x\in U$  there is $n_0>0$ and a neighbourhood $O\subset \mathbb{C}$ of $\pi f_c^{n_0}(x)$ such that \begin{equation}f_{c}^{n_0}(U) \supset \left (\pi|_{X} \right)^{-1}(O\cap \Lambda).  \end{equation} Since $\mathbf{f}_c$ is leo, there is  $n>0$ such that $\mathbf{f}_{c,\Lambda}^{n}(O\cap \Lambda)=\Lambda.$ Then $f_c^n$ maps $(\pi|_{X})^{-1}(O\cap \Lambda)$ onto $X.$ Therefore,
$f_c^{n_0 +n}(U) = X.$ This proves $f_c$ is leo. 
In the case of hyperbolic attractors $\Lambda$ of $\mathbf{f}_{c},$ we notice that  $\Lambda$  may be viewed as a hyperbolic repeller of $\mathbf{f}_c^{-1}.$ Thus the same arguments extend to this case.

Since $f_c: X \to X$ is expanding and  leo, $X$ is the closure of the periodic points of $f_c: X \to X.$ Hence $\Lambda$ is the closure of the union of all cycles contained in $\Lambda.$ Since every such cycle is either repelling (if $\Lambda$ is expanding) or attracting (if $\Lambda$ is attracting ), we conclude that $\Lambda\subset J_c$ or $\Lambda \subset J_c^*,$ respectively.
 \end{proof}

\subsection{Quasiconformal deformations of the universal cover }
Let $ \mathrm{e}^{iI} =\{\mathrm{e}^{it}: t\in I\}.$
We say that a family of curves $\gamma_c: \mathbb{R} \to \mathbb{C}$ parameterized on a region $U$ containing $0$  is a \emph{$K$-quasiconformal deformation of the universal cover of $\mathbb{S}^1$} if  \begin{itemize} \item[(a)]  $\gamma_0(t) = \mathrm{e}^{it}, \ t\in \mathbb{R};$
\item[(b)] for each $t\in \mathbb{R},$ the map $c\mapsto \gamma_{c}(t)$ is holomorphic on $U;$
\item[(c)] for every $t_0\in \mathbb{R},$ there is an open interval $I$ containing $t_0$ and a $K$-quasiconformal homeomorphism $\varphi:\mathrm{e}^{iI} \to L,$ for some $L \subset \mathbb{C},$  such that $\gamma(t)=\varphi(\mathrm{e}^{it})$ on $I.$
\end{itemize}

By the Schwarz Lemma (see the proof of  Theorem \ref{ofd}), property $(c)$  follows from $(a)$ and $(b)$ under the weaker assumption that $\varphi$ is injective. Notice that $K$ does not depend on $\varphi.$ We often indicate the family $\{\gamma_c\}$ by $\gamma: U \times \mathbb{R} \to \mathbb{C}.$

{Recall that $J_0$ is the unit circle. By Theorem \ref{int}, there is a continuous surjective map $g: \mathbb{R}\times \Sigma_q \to \mathcal{B}(\mathbb{S}^1)$ such that $\mathbb{R}\times \Sigma_q/ \sim_{g}$ is homeomorphic to the solenoid $\mathcal{B}(\mathbb{S}^1).$  Using this identification, every line   $\mathbb{R} \times \{\tau\}$ projects to a curve $\ell_{\tau}: \mathbb{R} \to \mathcal{B}(\mathbb{S}^1)$ in the quotient space.  The union of the images of such curves is the solenoid. For a fixed $\tau$ in $\Sigma_q,$ the holomorphic  motion $h_c: \mathcal{B}(\mathbb{S}^1) \to \mathbb{C}^2$  (Theorem \ref{giocco}) determines a family  $\gamma^{\tau}: U \times \mathbb{R} \to \mathbb{C}$ given by \begin{equation}\label{lfv} \gamma^{\tau}(c, t)= \pi \circ h_c \circ \ell_{\tau}(t),\end{equation}}

\nid {  \nid which by the following result is a family of $K$-quasiconformal deformations. 
\begin{thmm}[Parameterizations of the Julia set] \label{kks}There is a neighbourhood      $U \subset \mathbb{C}$ of zero and $K \geq 1$ such that every $\gamma^{\tau}$ in \eqref{lfv}  is $K$-quasiconformal deformation of the universal cover of $\mathbb{S}^1,$ parameterized on $U,$ with 
 \begin{equation}\label{sss} J(\mathbf{f}_c ) =\bigcup_{\tau \in \Sigma_{q}} \gamma^{\tau}_{c}(\mathbb{R}). \end{equation} Moreover,  $\mathcal{B}(J(\mathbf{f}_c))=J(f_c),$ and $J(\mathbf{f}_c)$ is a {leo} hyperbolic repeller, for every $c\in U.$  \end{thmm}}

{In  \eqref{sss} the family $\gamma^{\tau}_{0}:\mathbb{R} \to \mathbb{S}^1$ consists of just one curve (a great coincidence). Each $\tau$ determines a distinct quasiconformal deformation of the same curve. As a result we obtain the solenoidal structure apparent in Figure \ref{jdf}.  }

{We now turn to the proof of Theorem \ref{kks}.}

\begin{proof} Using the contraction of the local branches of $\mathbf{f}_0^{-1}$ on $\mathbb{S}^1$ we conclude that there is a constant $\ep >0$ such  that, for any $z\in \mathbb{S}^1$ and $n>0,$ there is a neighbourhood $U$ of $z$ such that \begin{equation}\mathbf{f}_0^n(U)\supset B_e(y, \ep),\end{equation} for any   $y\in \mathbf{f}_c^n(z).$ Since the  $\mathbf{f}_{0}^{n}(z)$ is determined by the roots $w$ of the equation $w^{q^n}=z^{p^n},$ which become dense in $\mathbb{S}^1$ as $n\to \infty,$ we conclude that $J_{0}=\mathbb{S}^1$ is a leo hyperbolic repeller of $\mathbf{f}_0.$  By Theorem \ref{jsd}, $\mathcal{B}(J_0)=J(f_{c_0})$ is expanding and leo under the action of the bundle map $f_{c_0}.$ By Theorem \ref{fds}, $f_c: J(f_c) \to J(f_c)$ is a leo expanding map of degree $q=\operatorname{card} f_c^{-1}(x),$ for every $x\in J(f_c).$  Since the bundle projection is a sort of semi-conjugacy, it follows that  all $q$ pre-images of any point in  $J(\mathbf{f}_c)$ are contained in $J(\mathbf{f}_c).$  Since $f_c$ is leo on $J(f_c),$ $\mathbf{f}_c$ is leo on $J(\mathbf{f}_c)= \pi J(f_c)$ (cf. Theorem \ref{fds}).  Since $J(\mathbf{f}_c)$ is a leo hyperbolic repeller, from Theorem \ref{jsd} we have $\mathcal{B}(J(\mathbf{f}_{c_0})) = J(f_c).$

      {
Using the notation of Theorem \ref{ofd}, for each fixed $\tau$, $\{\gamma_{c}^{\tau} \}_{c\in D_r}$ is a $K$-quasiconformal deformation of the universal cover of  $\mathbb{S}^1$ satisfying \eqref{sss}, 
  for every $c\in D_r.$ Notice that  $K$ does not depend on $c$ and $\tau$ provided $U$  is replaced by a region compactly contained in $U$.}      \end{proof}

    \subsection{Other  examples }\label{xvc} Are there more examples than $J_c$ for $c$ close to zero?  We now  give examples for a dense set of parameters $c$ in $\mathbb{S}^1$  based on results that will be published on a separate paper.  For such examples, dual Julia sets are stable Cantor sets obtained from Conformal Iterated Function Systems. (Such Cantor sets and the corresponding CIFS's are discussed in \cite{Carlos} with a slightly different notation).     
       
    \paragraph{Dual Julia sets.}   It can be proved that $J_0^*=\{0\}$ and $J_c^*$ is a Cantor set avoiding $0,$ for every $c$ in a punctured neighbourhood $U -\{0\}$ of zero. In this case the bundle projection is a homeomorphism from $\mathcal{B}(J_c^*)$ to $J_c^*$, and Theorem \ref{giocco} may be used to obtain a (non-branched) holomorphic motion  of $J_c^*,$ for every  $c\neq 0$ close to zero.

\paragraph{Julia sets.} For $d\geq 2,$ the branches of $(w-c)^{2} = z^{2d} $
 are $w=z^d +c $ and $w=-z^d + c.$  Let \begin{equation}\mathcal{S}_d =\{ \omega \in \mathbb{C}: \omega^{d-1}=-1\} \subset \mathbb{S}^1. \end{equation}  
Suppose $d\geq 2$ and  let $c_0\in \mathcal{S}_d.$ Let $R_d$ be the greatest real root of $x^d -x -1=0.$ Then $R_d \in (1,2)$ and \begin{equation}D_d(\infty)=\{z\in \mathbb{C}: |z| > R_d \} \end{equation} is an invariant basin of $\infty$ which contains $\{ z\in \mathbb{C}: \ |z| \geq 2\}.$   The first two iterates of $0$ yields a 2-cycle  and an orbit escaping to infinity: 

\begin{equation} 0 \mapsto c \mapsto c^n + c =0, \end{equation}
\begin{equation} 0 \mapsto c \mapsto  - c^n + c = -2c^n \in D_n(\infty). \end{equation}

Based on these facts, it is possible to show that: $J_{c_0}^*=\{0,c\}$ and $J_c^*$ is a Cantor set around the cycle $0\mapsto c \mapsto 0,$ for every $c$ in a punctured neighbourhood $V- \{c_0\}.$ For any $c\in V - \{c_0\},$ $J_c$ is hyperbolic and satisfies $(\star)$ of Remark \ref{dfg}. We therefore have a dense set of examples on the unit circle for which Theorems \ref{ofd} and \ref{fds} may be applied. For $d=2,$ and $c=-1 \in \mathcal{S}_{2},$ $J_c$ contains the Julia set of $z^2 -1$ (Basilica) and the Julia set of $-z^2 -1$ (a Cantor set, since the orbit of its critical point tends to infinity). So even for simple cases like $(w+1)^2= z^4$ the structure of the Julia set can be a very complicated fractal set contained in a disk of radius $2.$ Here we observe an intriguing phenomenon:  {\it each point of $\mathcal{S}_d$ lies at the center of a component of  the interior of the Multibrot set $\mathcal{M}_d$ of the family $z^d+c,$ }suggesting that the above examples may  be extended to  every point at the center of a component of the interior of $\mathcal{M}_d.$

\section{Proof of main Theorem \ref{giocco}}

\label{xxx}

 \nid We shall present a detailed proof only when $\Lambda_{c_0}$ is expanding, for every attracting hyperbolic set of $\mathbf{f}_{c_0}$  is an expanding hyperbolic set of $\mathbf{f}_{c_0}^{-1}. $  \subsection{Preliminary  assumptions}  Let $\mu(z)|dz|$ be a conformal metric on a neighbourhood $V$ of  $\Lambda_{c_0}$ satisfying \eqref{nice} $(a).$ Since any two conformal metrics are equivalent on compact sets, we choose an open set $V_0$ for which there is $\ell>1$ such that $\Lambda_{c_0} \subset V_0 \subset \overline{{V}}_0 \subset V$ and   \begin{equation} \label{poi}\frac{1}{\ell} d_e(x,y) \leq d_{\mu}(x,y) \leq \ell d_e(x,y),\end{equation} 
for any $x,y \in V_0,$ where $d_e$ denotes Euclidean distance and $d_{\mu}$ is the distance function coming from $\mu(z)|dz|.$
For a given $z\in \mathbb{C},$ let \begin{equation}\operatorname{sep}(\mathbf{f}_c,z)\end{equation} denote the infimum of all distances $d=|x-w|$  with $(w,x)\in \mathbf{f}_{c}(z) \times \mathbf{f}_{c}(z) \cup \mathbf{f}_{c}^{-1}(z) \times \mathbf{f}^{-1}_{c}(z)$ and  $w\neq x.$
Let $B_{\mu}(z,s)$ denote ball of radius $s$ and center $z\in V_0$ for the metric $d_{\mu}.$ The corresponding Euclidean ball is denoted by $B_{e}(z,s).$
Let \begin{equation}V(s)=\{z\in V_0: d_{\mu}(z, \Lambda_{c_0}) < s  \}. \end{equation}

\begin{lem}\label{(1)} {\it We may choose  $\ep>0$ and $\lambda \in (0,1)$ such that: \begin{enumerate} \item 
 $V(6\ep) \subset V_0$; \item   for every $z\in V_0$ and $c$ in \label{estimate} \label{sjl} 
    $U=\left\{c\in \mathbb{C}: 
  \ell|c-c_0| < \frac{\ep(1-\lambda)}{6}\right\}, $  we have  \begin{equation}\operatorname{sep}(\mathbf{f}_c,z) > 5\ep \ell;\end{equation}  \item  if $\varphi$ is a branch of $\mathbf{f}_c^{-1}$  at $z$ such that $\varphi(z) \in V_0,$ then 
$\|\varphi'(z)\| < \lambda,$  the domain of $\varphi$ includes $B_{\mu}(z, 6\ep)$ and 
         \begin{equation} d_{\mu}(\varphi(x), \varphi(y)) < \lambda d_{\mu}(x,y),\end{equation} for $x,y \in B_{\mu}(z,6\ep);$ 
         
         \item Let  $u\in U, \ z\in V_0$ and suppose  $\varphi$ is a branch of $\mathbf{f}_{u}^{-1}$ defined on $B_{\mu}(z, 6\ep)$ such that  $\varphi(z)\in V_0.$ If $v\in U,$  the translation  \begin{equation}\label{ijk}T_v(\varphi)(x)=  \varphi(x + u -v)\end{equation}  is a well defined holomorphic branch of $\mathbf{f}_v^{-1}$ whose domain includes  $B_{\mu}(z, 5\ep).$ 
\item  Suppose $u,v\in U.$ Let  $z\in V_0$  and let $\varphi$ be a branch of $\mathbf{f}_{u}^{-1}$  defined at $z$ such that $\varphi(z) \in V_0.$ Then 
\begin{equation}d_{\mu}(T_v\varphi(x), T_v\varphi(y)) < \lambda d_{\mu}(x,y), \end{equation} 
for every $x,y\in B_{\mu}(z, 5\ep).$ 
         \end{enumerate}
 }
  \end{lem}

\begin{proof} Follows from the metric properties of  the hyperbolic set. 
\end{proof}
 \subsection{The construction of $h_c:\mathcal{B}(\Lambda_{c_0}) \to \mathbb{C}^2$} \label{bmw} 
 In the following  proof a large role is played by a classic shadowing lemma, adapted to the purposes of this paper. (Many proofs of stability rely on shadowing arguments.)  
   Let $U$  be as in Lemma \ref{(1)} (2)  and let $T_v$ be as in  \eqref{ijk}.    Let $u\in U$ and let $\{z_i\}_{0}^{\infty} \subset V_0$ be an orbit of $\mathbf{f}_u.$ By  Lemma \ref{(1)} (3), for each $i>0$   there is a  branch $\varphi_i$ of $\mathbf{f}_{u}^{-1}$ defined on $B_{\mu}(z_i, 6\ep)$ such that $\varphi_i(z_i)=z_{i-1}.$We are going to construct a double sequence of holomorphic functions $g_{ij}: U\to B_{\mu}(z_j, \ep/3),$ with $0\leq j \leq i.$  The sequence $g_{ij}$ is inductively defined by  \begin{equation}g_{nn}(v) = z_n, \ g_{n(i-1)}(v) = T_v(\varphi_i)(g_{ni}(v)),\end{equation} for $v\in U.$  From Lemma \ref{(1)} (4), this definition makes sense because $g_{ni}(v)$ is always in the domain of $T_v(\varphi_i).$    As a matter of fact,  since $T_v(\varphi_i)$ contracts distances by  $\lambda,$
 \begin{equation}\label{xcm}d_{\mu}(g_{nj}(v), z_j) \leq \sum_{k=1}^{n-j} \lambda^{k} \ell|v-u| < \frac{\ep}{3},\end{equation}
for every  $v\in U.$ As a bounded sequence of holomorphic functions, the family $\{g_{nj}\}_n$ is normal, for each fixed $j.$ Hence there is an infinite set $\mathcal{N}_j \subset \mathbb{N}$ such that $g_{nj}(c)$ converges locally uniformly to a holomorphic function $g_j:U\to B_{\mu}(z_j, \ep/3)$ as $n\to \infty$ and $n\in \mathcal{N}_j.$ We may choose these sets so that  the first element of $\mathcal{N}_j$ is greater than the $jth$ element of $\mathcal{N}_{j-1},$ and  $\mathcal{N}_{j+1} \subset \mathcal{N}_{j}.$   Using the standard diagonal method we can select an infinite set $\mathcal{N}$ which is contained in the intersection of all $\mathcal{N}_j.$ Since \begin{equation}g_{n(j-1)}(v) \xrightarrow{\mathbf{f}_v} g_{nj}(v) \end{equation} and both sequences converge as $n\to \infty, \ n\in \mathcal{N},$ we have $ g_{j-1}(v) \xrightarrow{\mathbf{f}_v} g_j(v),$ for every $j.$ 

\begin{lem}[Shadowing]\label{sdw}{\it  The sequence $w_j=g_j(v)$ is the unique orbit of $\mathbf{f}_v$ such that $d_{\mu}(w_j, z_j) < \ep,$ for every $j.$ If $z_j$ is periodic, then so is $w_j$ (with the same period). }\end{lem}
 
\begin{proof}
If $\{\tilde{w}_j\}$ is another orbit of $\mathbf{f}_v$ with the same property, then $d_{\mu}(w_j, \tilde{w}_j) < 2\ep.$  With the help of Lemma \ref{(1)} (2),  we can prove that the points are related by the same branch: $\varphi_j(w_j)=w_{j-1}$ and $\varphi_j(\tilde{w}_j) = \tilde{w}_{j-1}.$ Hence
\begin{equation} \label{opq} d_{\mu}(w_j, \tilde{w}_j) \leq \lambda^{k} d_{\mu}(w_{j+k}, \tilde{w}_{j+k})\leq \lambda^{k}2\ep \to 0,  \  \textrm{as} \  k \to \infty.  \end{equation}
\end{proof}

\begin{defi}Let $h_c: \mathcal{B}(\Lambda_{c_0}) \to \mathbb{C}^2$ be given by  \begin{equation}\label{ios}h_c(x)=\left( g_0(c), r\sum_{k=1}^{\infty} \delta^{k-1} g_{k}(c)\right),\end{equation} where $w_j=g_j(c)$ is the unique orbit of $\mathbf{f}_c$ such that $d_{\mu}(\pi f_{c_0}^n(x), w_n) < \ep,$ for every $n.$ \end{defi}  By definition, $c\mapsto h_c(x)$ is holomorphic on $U,$ for any  $x\in \mathcal{B}(\Lambda_{c_0}).$

\subsection{The function $h_c$ is a topological conjugacy onto its image. }  The continuous dependence of the shadowing (which can be proved with an argument similar to \eqref{opq}) establishes the continuity of $h_c.$

Let $B_c= h_c(\mathcal{B}(\Lambda_{c_0})).$ The inverse $h_c^{-1}: B_c\to \mathcal{B}(\Lambda_{c_0})$  is determined by the uniqueness property of  Lemma \ref{sdw}, which makes  the definitions of $h_c^{-1}$  and $h_c$ essentially the same (but on different spaces). Hence the same continuous dependence of the shadowing used for $h_c$ also proves that $h_c^{-1}$ is continuous on $B_c.$ The uniqueness of the shadowing  and Theorem \ref{semiconjugacy} $(d)$ makes it possible to prove that $B_c$ is invariant under the bundle map $f_c,$ 
and $h_c \circ f_{c_0}= f_c \circ h_c$ on $\mathcal{B}(\Lambda_{c_0}).$  This conjugacy equation also proves that $\Lambda_c=\pi h_c\mathcal{B}(\Lambda_{c_0})$ belongs to   $\inv(A_r, c).$ By \eqref{xcm},  we have \begin{equation}\label{lfc}\disth(\Lambda_c, \Lambda_{c_0}) < \frac{\ell}{1-\lambda}|c-c_0|  < \frac{\ep}{3}, \end{equation}
 for every $c\in U.$ By Lemma \ref{(1)} (1) and (3), this establishes the continuity and hyperbolicity of $\Lambda_c$  for $c\in U.$ Since $\Lambda_c$ is hyperbolic, it makes sense to define  $\mathcal{B}(\Lambda_c).$ From the definition of $h_c,$ we have  $h_c(\mathcal{B}(\Lambda_{c_0})) \subset \mathcal{B}(\Lambda_c).$

\subsection{On the proofs of $(d), (e)$ and $(f)$ of Theorem \ref{giocco}} The Lipschitz property $(d)$ follows from \eqref{xcm} and the fact that $d_e$ and $d_{\mu}$ are equivalent on $V_0.$    Properties $(e)$ and $(f)$ follow from Lemma \ref{sdw}, with $O=V_0.$  The map $g: U\times \mathcal{B}(\Lambda_c) \to \mathbb{C}^2$ of $(f)$  is determined by  uniqueness of shadowing.

 \subsection{Uniqueness}

Assume $g,h:U\times \mathcal{B}(\Lambda_{c_0}) \to \mathbb{C}^{2}$ are two holomorphic motions satisfying Theorem \ref{giocco} $(a)$-$(d).$ Fix $x\in \mathcal{B}(\Lambda_{c_0})$ and let \begin{equation} W_x=\left\{c\in U: g_c(x)=h_c(x)\right\}.\end{equation} This set is closed and contains $c_0.$ We are going to show that $W_x$ is open with the help of Theorem \ref{semiconjugacy}.
If $z_n(c) = f_c^n(g_c(x))$ and $w_n(c)=f_c^n(h_c(x)),$ then \begin{equation}|z_i(u) - z_i(v)| \leq C_0 |u-v|, \ \ |w_i(u) - w_i(v)| \leq C'_0|u-v|, \end{equation}
for every $u,v \in U.$ Here we have used the estimate $(d)$ and the conjugacy property $(c).$ Let $u\in W_x.$ Since $d_e$ and $d_{\mu}$ are equivalent on a compact neighbourhood of $\Lambda_u,$ there is an open set $O\subset U$ containing $u$ such that $d_{\mu}(z_i(u), z_i(v)) < \ep$ and  $ d_{\mu}(w_i(u), w_i(v)) < \ep,$   for every $i\geq 0$ and  $v\in O.$ By Lemma \ref{sdw}, we have $O\subset W_x.$ This proves $W_x$ nonempty, open and closed in $U$. Since $U$ is connected, we have $W_x=U$ and $g=h.$

\nid The proof of Theorem \ref{giocco}  is complete.

\section{Acknowledgments}
 This work was partially supported by grant 2010/17397-2  S\~ao Paulo Research Foundation (FAPESP) and grant CNPq  232706/2014-0.
C.S. would like to thank Prof. Shaun Bullett and Prof. Christopher Penrose for the comments on an earlier version of these results presented in a seminar at Queen Mary University of London. 
An anonymous referee read the whole manuscript and gave us extensive comments for which we are extremely grateful.

Approximately half part of this paper was obtained from  \cite{Carlos}; the other half was  developed during C.S.'s postdoctoral stage at Imperial College London (Academic Year 2015/2016). C.S. is immensely grateful to Prof. Sebastian van Strien.

\bibliographystyle{alpha}
\bibliography{oi}

\end{document}